\documentclass[a4paper,11pt,reqno,twoside]{amsart}

\usepackage{amsmath}
\usepackage{amsfonts}
\usepackage{amssymb}
\usepackage{amsthm}
\usepackage{color}
\usepackage{ifpdf}
\usepackage{array}
\usepackage{mathtools}
\usepackage{url}
\usepackage{multirow,bigdelim}

\numberwithin{equation}{section}
 
\newtheorem{theorem}{Theorem}[section] 
\newtheorem{lemma}{Lemma}[section] 
\newtheorem{proposition}{Proposition}[section] 
\newtheorem{corollary}{Corollary}[section] 
\newtheorem{remark}{Remark}[section]

\newcommand*{\C}{\mathbb{C}}%.............................C
\newcommand*{\R}{\mathbb{R}}%.............................R
\newcommand*{\Z}{\mathbb{Z}}%.............................Z
\newcommand{\comment}[1]{}
\title[An inverse problem for a class of lacunary canonical systems with diagonal Hamiltonian]%
      {An inverse problem for a class of lacunary canonical systems with diagonal Hamiltonian} 
\author[M. Suzuki]{Masatoshi Suzuki}
\subjclass[2010]{34A55, 31A10, 34L40}
\keywords{}
\AtBeginDocument{%
\begin{abstract}
Hamiltonians are 2-by-2 positive semidefinite real symmetric matrix-valued functions satisfying certain conditions. 
In this paper, we solve 
the inverse problem 
for which recovers a Hamiltonian from the solution of a first-order system 
attached to a given Hamiltonian, 
consisting of ordinary differential equations parametrized 
by a set of complex numbers,  
under certain conditions for the solutions. 
This inverse problem is a generalization 
of the inverse problem for two-dimensional canonical systems. 
\end{abstract}
\maketitle
}
\begin{document}
%..........................................

%
%%%%%%%%%%%%%%%%%%%%%%%%%%%%%%%%%%%%%%%%%%%%%%%%%%%%%%%%%%%%%%%%%%%%%%%%%%%%%%%%%%%%%%%%
%
\section{Introduction} \label{section_1} 
%
%%%%%%%%%%%%%%%%%%%%%%%%%%%%%%%%%%%%%%%%%%%%%%%%%%%%%%%%%%%%%%%%%%%%%%%%%%%%%%%%%%%%%%%%
%

In this paper, we generalize the theory on the inverse problem for 
a class of two-dimensional canonical systems in \cite{Su19_1}. 
%together with some simplifications of argument.
%
%To start with, we explain that canonical systems generate functions of $\overline{\mathbb{HB}}$.  
%Let ${\rm Sym}_2(\R)$ be the set of all $2 \times 2$ real symmetric matrices. 
%A ${\rm Sym}_2(\R)$-valued function $H$ defined on 
A $2 \times 2$ real symmetric matrix-valued function $H$ defined on 
an interval $I=[t_0,t_1)$ $(-\infty<t_0<t_1 \leq \infty)$ is called a {\it Hamiltonian} 
if $H(t)$ is positive semidefinite for almost all $t \in I$, 
$H$ is not identically zero on any subset of $I$ with positive Lebesgue measure, 
and $H$ belongs to $L^1([t_0,c),\R^{2\times 2})$ for any $t_0<c<t_1$.
The first-order system 
\begin{equation} \label{s105} 
-\frac{d}{dt}
\begin{bmatrix}
A(t,z) \\ B(t,z)
\end{bmatrix}
= z 
\begin{bmatrix}
0 & -1 \\ 1 & 0
\end{bmatrix}
H(t)
\begin{bmatrix}
A(t,z) \\ B(t,z)
\end{bmatrix}
\end{equation}
associated with a Hamiltonian $H$ on an interval $I$ parametrized by {\it all} $z \in \C$ 
is called a {\it canonical system} on $I$. 
If the range of $z$ is restricted to a subset of $\C$, 
such a system is called a {\it lacunary} canonical system (in this paper).
Note that definition \eqref{s105} is different from the usual definition with a minus sign on the right 
for consistency with \cite{Su19_1}.  
A typical source of Hamiltonians is entire functions of the {\it Hermite--Biehler class}, 
which is the set $\overline{\mathbb{HB}}$ of all entire functions satisfying 
\begin{equation} \label{s103}
|E^\sharp(z)| < |E(z)|  \quad \text{for all $z \in \C_+$}, 
\end{equation}
and the subset $\mathbb{HB}$ of $\overline{\mathbb{HB}}$ 
consisting of $E$ such that $E(z)\not=0$ for any $z \in \R$, 
where $\C_+ = \{z\,|\,\Im(z) > 0\}$ is the upper half-plane and 
$F^\sharp(z):=\overline{F(\bar{z})}$, the notation is often used in this paper.  
Every $E \in \mathbb{HB}$ generates a de Branges space $\mathcal{H}(E)$ 
which is a reproducing kernel Hilbert space of entire functions. 
% with some additional properties. 
%
Every de Branges space $\mathcal{H}(E)$ has a unique maximal chain 
of de Branges subspaces $\mathcal{H}(E_t)$ parametrized by $t$ in an interval $I$ 
such that $\mathcal{H}(E_t)$ contained isometrically in $\mathcal{H}(E)$ for all regular $t \in I$ 
(see \cite[Section 1]{Su19_1} for the definition of regular points). 
For the generating functions $E_t$, $A_t=(E_t+E_t^\sharp)/2$ and $B_t=i(E_t-E_t^\sharp)/2$ 
satisfy a canonical system on the interval $I$ 
associated with some Hamiltonian $H$.  
Such Hamiltonian is called the {\it structure Hamiltonian} of $\mathcal{H}(E)$. 
Recently, a complete characterization of structure Hamiltonians of de Branges spaces 
is obtained by Romanov--Woracek~\cite{RomWor19}. 
The inverse problem for recovering the structure Hamiltonian from given $E \in \mathbb{HB}$ 
has been studied by many authors; 
see the references cited just below \cite[Theorem dB]{Su19_1}, 
for example. 
%see Winkler \cite{Win95}, 
%Remling \cite{Rem18}, Romanov \cite{Rom14}, Suzuki \cite{Su18_2}, 
%and references therein, for example. 
Also, in \cite{Su19_1}, a method of recovering $H$ 
from $E \in \mathbb{HB}$ satisfying some specific conditions was discussed.

However, if $E$ does not necessarily belong to $\overline{\mathbb{HB}}$, 
nothing can be said in general about whether Hamiltonian $H$ exist 
such that a solution of the first-order system \eqref{s105} 
associated with $H$ recovers $E$.  
%it is not known whether or not Hamiltonian can be obtained from $E$. 
%
In this paper, we prove that if we assume several conditions on a function $E$, 
a Hamiltonian $H_E$ is obtained from $E$ by an explicit way of the construction, 
and a solution of a lacunary canonical system associate with $H_E$ 
recovers $E$, 
%and a solution of the first-order system \eqref{s105} associate with $H_E$ 
%parametrized by certain {\it subset} of complex numbers 
%recovers $E$, 
even though $E$ does not necessarily belong to the Hermite--Biehler class. 
Those conditions introduced below may look artificial, 
but they naturally arise from number theory; 
see the final part of the introduction. 
We should emphasize that an example there shows that 
our method of constructing a Hamiltonian 
actually provides a Hamiltonian that is {\it not} a structure Hamiltonian of de Branges spaces.
Also, as explained after Theorem 1.1 in \cite{Su19_1}, 
the Hamiltonian obtained by our method is not ``trivial'' one.  

As usual, we denote by 
\begin{equation} \label{eq_Fourier}
({\mathsf F}f)(z) 
 = \int_{-\infty}^{\infty} f(x) \, e^{ixz} \, dx, \quad 
({\mathsf F}^{-1} g)(z) 
 = \frac{1}{2\pi}\int_{-\infty}^{\infty} g(u) \, e^{-ixu} \, du
\end{equation}
the Fourier integral and inverse Fourier integral, respectively, 
and use the Vinogradov symbol ``$\ll$'' to estimate a function. 
%as well as the Landau symbol ``$O$''. 
\smallskip

We introduce the following six conditions for a function $E$: 
\begin{enumerate}
\item[(K1)] There exists $c>0$ and a discrete subset $0\not\in\mathcal{Z} \subset \C$ 
(which is possibility empty or infinite) 
of the  horizontal strip $ -c \leq \Im z \leq c$ such that 
it is closed under the complex conjugation $z \mapsto \bar{z}$ 
and the negation $z \mapsto -z$ 
and $E$ is analytic in $\C\setminus\mathcal{Z}$ 
and that $E$ satisfies  $E^\sharp(z)=E(-z)$ for $z \in \C\setminus\mathcal{Z}$; 
%{\color{red} normalize $E(0)=1$ ?} 
%
\item[(K2)] There exists a real-valued continuous function $K$ defined on the real line 
such that $|K(x)| \ll \exp(c|x|)$ as $|x| \to \infty$ 
and that $E^\sharp(z)/E(z)= (\mathsf{F}K)(z)$ holds for $\Im(z)>c$, where $c$ is the constant in (K1).
\item[(K3)] $K$ vanishes on the negative real line $(-\infty,0)$;  
\item[(K4)] $K$ is continuously differentiable outside a discrete subset $\Lambda \subset \R$ 
and the derivative $K'$ belongs to $L_{\rm loc}^1(\R)$; 
% is locally integrable on $\R$. {\color{red} $K' \in L_{\rm loc}^1(\R)$ のことだから, 絶対値は不要.} 
%
\item[(K5)] There exists $0<\tau \leq \infty$ such that both $\pm 1$ are not eigenvalues of 
the operator 
$
\mathsf{K}[t]: f(x) \mapsto
\mathsf{1}_{(-\infty,t)}(x)\, \int_{-\infty}^{t} K(x+y) \, f(y) \, dy
$
 on $L^2(-\infty,t)$ for every $0 \leq t < \tau$, 
where $\mathsf{1}_A$ is the characteristic function of $A \subset \R$; 
\item[(K6)] The function $E^\sharp(z)/E(z)$ can not be expressed as a ratio of two entire functions of exponential type.
\end{enumerate}
Condition (K1) is more general than that in \cite{Su19_1}, 
because it is allowed that $E$ has an essential singularity in $\mathcal{Z}$. 
Conditions (K2)$\sim$(K5) are the same as in \cite{Su19_1}. 
Condition (K6) is added in this paper for convenience, but it is rarely used. 
The operators $\mathsf{K}[t]$ in (K5) are well-defined, 
Hilbert-Schmidt, and self-adjoint under (K2) and (K3) (\cite[Section 1]{Su19_1}). 
Therefore, the spectrum of $\mathsf{K}[t]$ consists only of real eigenvalues of finite multiplicity and $0$. 
Moreover, $\mathsf{K}[t]=0$ for $t \leq 0$ under (K2) and (K3), 
hence the requirement for eigenvalues of $\mathsf{K}[t]$ in (K5) is trivial for $t \leq 0$. 
For a positive $t$, 
it is not difficult to show that (K5) holds for a sufficiently small  $\tau>0$ under (K2) and (K3).
%({\color{red} 核の連続性からK[t]のHilbert-Schmidtノルムが $t\to 0$ のとき$\to 0$であることと, 
%作用素ノルムが最大固有値の絶対値に等しいこと，及び作用素ノルムがHilbert-Schmidtノルムで 
%bound されることから, (K5) holds if $\tau>0$ is sufficiently small})
%
The set of functions satisfying (K1)$\sim$(K6) is not empty but a large; see the final part of the introduction. % and Section \ref{section_RH}. 
\medskip

Now we assume that $E$ satisfies (K1)$\sim$(K5) and define 
\begin{equation} \label{s333}
H(t):= \begin{bmatrix} 1/\gamma(t) & 0 \\ 0 & \gamma(t) \end{bmatrix}, \qquad 
\gamma(t):=m(t)^2, \qquad 
m(t):=\frac{\det(1+\mathsf{K}[t])}{\det(1-\mathsf{K}[t])},  
\end{equation}
where ``$\det$'' stands for the Fredholm determinant. 
%We have $m(t)=1$ and $H(t)$ is the identity matrix if $t \leq 0$. 
Then $\gamma(t)$ is a continuous positive real-valued function on $I_\tau=(-\infty,\tau)$ 
(Theorem \ref{thm_5} and Proposition \ref{prop_190620_1}). 
Therefore $H$ is a Hamiltonian on $I_\tau$ consisting of continuous functions. 
%
%%%
\comment{
For a Hamiltonian $H$ on $I$, 
a nonempty open subinterval $J$ of $I$ is called {\it $H$-indivisible}, 
if the equality  
$H(t)=h(t)\xi_\theta{}^{\rm t}\xi_\theta$, 
 $\xi_\theta={}^{\rm t}[\cos\theta,\sin\theta]$, 
holds on $J$ for some positive scalar function $h(t)$ 
and some fixed angle $0 \leq \theta < \pi$. 
We have $\det H(t)=0$ and 
${\rm rank}\,H(t)=1$ on an $H$-indivisible interval. 
A point $t \in I$ is called {\it regular} 
if it is not an inner point of any $H$-indivisible interval, 
otherwise, $t$ is called {\it singular}. 
A Hamiltonian defined by \eqref{s335} and \eqref{s333} 
has no $H$-indivisible intervals, that is, all points of $[0,\tau)$ are regular. 
}

The direct problem for the lacunary canonical system on $I_\tau$ 
associated with $H$ of \eqref{s333} 
and the vector ${}^{t}[A(z)~B(z)]$ as intitial condition at $t=0$ 
can be solved explicitly as follows, 
where $A$ and $B$ are defined by \eqref{s104} below. 
Then the solution of the lacunary canonical system on $I_\tau$ 
associated with $H$ of \eqref{s333} 
recovers the original $E$.  
The solution of the first-order system is 
explicitly described by using the unique solutions of the integral equations 
\begin{equation} \label{s304_1}
\Phi(t,x) + \int_{-\infty}^{t} K(x+y) \Phi(t,y) \, dy = 1,  
\end{equation}
\begin{equation} \label{s304_2}
\Psi(t,x) - \int_{-\infty}^{t} K(x+y) \Psi(t,y) \, dy = 1. 
\end{equation}

\begin{theorem} \label{thm_1}
Let $E$ be a function satisfying (K1)$\sim$(K5), and define $A$ and $B$ by
\begin{equation} \label{s104}
A(z) := \frac{1}{2}(E(z)+E^\sharp(z)) \quad \text{and} \quad 
B(z) := \frac{i}{2}(E(z)-E^\sharp(z)).
\end{equation}
Let $H$ be the Hamiltonian on $I_\tau=(-\infty,\tau)$ defined by \eqref{s333}.  
Let $\Phi(t,x)$ and $\Psi(t,x)$ be the unique solutions of \eqref{s304_1} and \eqref{s304_2}, respectively.  
Define $A(t,z)$ and $B(t,z)$ by 
\begin{equation} \label{eq_190625_1}
\aligned
A(t,z)
& = - \frac{iz}{2}E(z) \int_{t}^{\infty} \Psi(t,x) e^{izx} \, dx,
\\ 
-i B(t,z) 
& = - \frac{iz}{2}E(z) \int_{t}^{\infty} \Phi(t,x) e^{izx} \, dx. 
\endaligned
\end{equation}
Then,  
\begin{enumerate}
\item for each $t \in I_\tau$, $A(t,z)$ and $B(t,z)$ are well-defined for $\Im(z)>c$ 
and extend to analytic functions on $\C\setminus\mathcal{Z}$ 
satisfying 
$A(t,-z)=A(t,z)$, $B(t,-z)=-B(t,z)$, 
$A(t,z)=\overline{A(t,\bar{z})}$, and $B(t,z)=\overline{B(t,\bar{z})}$, 
%
%\item for each $z \in \C\setminus\mathcal{Z}$, $A(t,z)$ and $B(t,z)$ are continuous and piecewise continuously differentiable function of $t$, 
\item for each $z \in \C\setminus\mathcal{Z}$, 
$A(t,z)$ and $B(t,z)$ are continuously differentiable with respect to $t$, 
\item $A(t,z)$ and $B(t,z)$ solves the first-order system \eqref{s105} associated with $H$ on $I_\tau$ 
for every $z \in \C\setminus\mathcal{Z}$
%and parametrized by $z \in \C\setminus\mathcal{Z}$, %attached to $H_L^{\omega,\nu}$ 
%
\item $A(z)=A(0,z)$, $B(z)=B(0,z)$ and $E(z)=A(0,z) - i B(0,z)$.
\end{enumerate} 
\end{theorem}

Theorem \ref{thm_1} is a generalization of 
\cite[Theorem 1.1]{Su19_1} by Lemma \ref{lem_5_1} below. 

Let $H^\infty=H^\infty(\C_+)$ be the space of all bounded analytic functions in $\C_+$. 
A function $\theta \in H^\infty$  
is called an {\it inner function} in $\C_+$ 
if $\lim_{y \to 0+}|\theta(x+iy)|=1$ for almost all $x \in \R$ 
with respect to the Lebesgue measure. 
We define 
\begin{equation*}
\Theta(z)=\Theta_E(z):=\frac{E^\sharp(z)}{E(z)}
\end{equation*} 
under (K1). 
Then, $\Theta(0)=1$,  
\begin{equation} \label{105}
\Theta(z)\Theta(-z)=1 \quad \text{for} \quad z \in \C \setminus \mathcal{Z},
\quad \text{and} \quad 
|\Theta(u)|=1 \quad \text{for} \quad u \in \R \setminus \mathcal{Z}
\end{equation}
by definition, 
but of course $\Theta_E$ is not inner in general. 
Now we clarify the meaning of the condition (K5) 
in order to relate $H$ of \eqref{s333} 
to the structure Hamiltonian of the de Branges space $\mathcal{H}(E)$ 
when $E$ is an entire function. 

\begin{theorem} \label{thm_2} 
Assume that $E$ satisfies (K1)$\sim$(K3), and (K5) with $\tau=\infty$. 
Then $\Theta_E$ is an inner function in $\C_+$. 
In particular, possible singularities $\mathcal{Z}$ are real. 
\end{theorem}
\begin{remark} 
Compare this with \cite[Theorem 2.4]{Su19_2}, 
where some additional conditions are assumed in order to conclude that 
$\Theta_E$ is an inner function in $\C_+$,  
for $E$ arising from $L$-functions in the Selberg class.  
\end{remark}

If $E \in \mathbb{HB}$, $\Theta_E$ is an inner function in $\C_+$.  
Theorem \ref{thm_2} 
shows that (K5) with $\tau=\infty$ plays the role of the condition $E \in \mathbb{HB}$ 
for entire functions $E$; $\mathcal{Z}=\emptyset$. 
On the other hand, it is known that if $\theta$ is an inner function and meromorphic in $\C_+$, 
there exists $E \in \mathbb{HB}$ such that $\theta=E^\sharp/E$ (\cite[Sections 2.3 and 2.4]{MR2016246}). 
However, the existence of $\tau>0$ in (K5) is not obvious 
even if we assume that $\Theta_E$ is inner in $\C_+$. 
Therefore, for the converse of Theorem \ref{thm_2}, 
we require (K6). 

\begin{theorem} \label{thm_3} Assume that $E$ satisfies (K1)$\sim$(K3), and (K6). 
In addition, assume that $\Theta_E$ is an inner function in $\C_+$. 
Then (K5) holds for $\tau=\infty$.  
\end{theorem}
As we will see in the proof, 
Theorem \ref{thm_3} is essentially the same as \cite[Theorem 5.2]{Su19_1}. 
Theorems \ref{thm_1}, \ref{thm_2}, and \ref{thm_3} 
emphasize the importance of the function $m(t)$. 
The following simple formula, which was \cite[Theorem 1\, (1.9)]{Su18} 
for special $E$, 
is interesting from both theoretical and computational aspects.

\begin{theorem} \label{thm_5} Assume that $E$ satisfies (K2)$\sim$(K5). Then,  
\begin{equation} \label{formula_01}
m(t) = \frac{1}{\Phi(t,t)} = \Psi(t,t)
\end{equation}
holds for every $t \in \R$. 
\end{theorem}

See Propositions \ref{prop_190620_2} and \ref{prop_190709_1} for other formulas of $m(t)$.
If $E$ satisfies (K1)$\sim$(K5) and $\Theta_E$ is an inner function in $\C_+$, 
$\mathsf{K}=\lim_{t \to \infty}\mathsf{K}[t]$ defines a bounded operator on $L^2(\R)$ 
(Lemma \ref{lem_3_2}), and the Fourier transform 
$\mathsf{F}(\mathcal{V}_t)$ of the space 
$\mathcal{V}_t=L^2(t,\infty) \cap \mathsf{K}L^2(t,\infty)$ 
forms a reproducing kernel Hilbert space for each $0 \leq t <\tau$ 
(Section \ref{section_6}). 

\begin{theorem} \label{thm_4} The following statements hold.
\begin{enumerate}
\item Assume that $E$ satisfies (K1)$\sim$(K5) and that $\Theta_E$ is an inner function in $\C_+$. 
Let $A(t,z)$ and $B(t,z)$ be as in Theorem \ref{thm_1}, 
and let $j(t;z,w)$ be the reproducing kernel of $\mathsf{F}(\mathcal{V}_t)$ for $0 \leq t<\tau$. 
Then, 
\begin{equation} \label{s106}
j(t;z,w) 
 = \frac{1}{\overline{E(z)}E(w)}\cdot \frac{\overline{A(t,z)}B(t,w)-A(t,w)\overline{B(t,z)}}{\pi(w-\bar{z})}, 
\end{equation}
and $j(t;z,z) \not \equiv 0$ as a function of $z \in \C_+$ for any $0 \leq t<\tau$. 
\item Assume that $E$ satisfies (K1)$\sim$(K5) with $\tau=\infty$.   
Then, ($\Theta_E$ is an inner function in $\C_+$ and) $\displaystyle{\lim_{t \to \infty}j(t;z,w) = 0}$ 
for every $z,w \in \C_+$.
\end{enumerate}
\end{theorem}

Theorem \ref{thm_4} is a generalization of \cite[Theorem 1.2]{Su19_1} 
and shows that $H$ of \eqref{s333} provides the structure Hamiltonian 
of the de Branges space $\mathcal{H}(E)$ 
if $E \in \mathbb{HB}$ satisfies (K1)$\sim$(K3) and (K6) 
by \cite[Theorem 40]{MR0229011}. 
In this sense,  we could call $H$ of \eqref{s333} 
the ``structural Hamiltonian'' of the family of spaces 
$\{\mathsf{F}(\mathcal{V}_t)\}_{0 \leq t <\infty}$ 
even though $\Theta_E$ is inner but $E \not\in \mathbb{HB}$. 
If $\Theta_E$ is inner, the spaces $\mathsf{F}(\mathcal{V}_t)$ 
are subspaces of the model space $\mathcal{K}(\Theta_E)$ (see Section \ref{section_6}). 
In such a case, Theorems \ref{thm_1} and \ref{thm_4} 
solve the ``inverse problem'' to find the ``structure Hamiltonian'' 
of  $\{\mathsf{F}(\mathcal{V}_t)\}_{0 \leq t <\infty} \subset \mathcal{K}(\Theta_E)$ 
from the generator $\Theta_E$. 
We are not sure how useful the above extensions of known classical situation for the Hermite-Biehler functions are in general, but at least they have motivations related to the Riemann zeta function, as described below.

The basic idea for achieving the above results originates from the work of J.-F. Burnol 
\cite{Burnol2011} (and \cite{Burnol2002, Burnol2004, Burnol2007}) 
as well as \cite{Su12, Su19_1, Su19_2, Su18}. 
However, in this paper, the method used in \cite{Burnol2011} for $\Gamma(1-s)/\Gamma(s)$ 
standing on the theory of Hankel transforms 
is axiomatized, reorganized and generalized, and some arguments are simplified.
\bigskip

Before closing the introduction, 
we mention a few examples of functions $E$ satisfying conditions (K1)$\sim$(K6).
Let $\zeta(s)$ be the Riemann zeta function, 
and let $\xi(s)=s(s-1)\pi^{-s/2}\Gamma(s/2)\zeta(s)$. 
Then $\xi(s)$ is an entire function taking real-values on the critical line $\Re(s)=1/2$ 
and the real line such that 
the zeros coincide with nontrivial zeros of $\zeta(s)$. 
We put
\begin{equation} \label{190709_1}
E_{n}(z) = \xi\left(\frac{1}{2}+\frac{2}{n}-iz\right)^n, 
\quad 
E_{\Join}(z) = \exp\left( 2 \frac{\xi'}{\xi}\left(\frac{1}{2}-iz\right) \right)  
\end{equation}
for $n \in \Z_{>0}$. 
Then, $E_{\Join}(z)=\lim_{n\to\infty}E_n(z)/\xi(1/2-iz)^n$, 
and it is proved that $E_n$ (resp. $E_{\Join}$) satisfies 
the conditions (K1)$\sim$(K6) 
in \cite[Propositions 4.1 and 4.2 and Lemma 5.1]{Su19_2} (resp. \cite[Theorem 2]{Su18}), 
where $\tau>0$ in (K5) is small if unconditional, 
and $\tau=\infty$ if the Riemann hypothesis is assumed; 
all zeros of $\xi(s)$ lie on the critical line. 
In \cite{Su19_2}, 
many examples of $E$ satisfying the conditions (K1)$\sim$(K6) are made 
from $L$-functions in the Selberg class. 

The function $\xi(s)$ has no zeros in $\Re(s)>\tfrac{1}{2}+\tfrac{2}{N}$ 
if and only if $E_n$ belongs to $\mathbb{HB}$ for each $n \geq N$ 
(\cite[Theorem 6.1]{Su19_2}). 
In particular, the de Branges space $\mathcal{H}(E_n)$ is defined for each $n \in \Z_{>0}$ 
under the Riemann hypothesis, 
and its structure Hamiltonian $H_n$ is constructed in \cite{Su19_1, Su19_2}. 
Therefore, it is natural to ask about the limit behavior of $H_n$ as $n\to\infty$. 
However, $\lim_{n\to\infty}E_n$ does not make sense, 
and $E_{\Join}$ is no longer an entire function, 
because $E_{\Join}$ has an essential singularity at a zero of $\xi(1/2-iz)$. 
Therefore, the method constructing $H_n$ in \cite{Su19_1, Su19_2} can not be applied to $E_{\Join}$. 
This is the main reason why we generalized condition (K1) as above in this paper.
\bigskip

As mentioned above, 
Theorems \ref{thm_1} and \ref{thm_4} are generalizations of some results in \cite{Su19_1}, 
and Theorem \ref{thm_5} is a generalization of a result in \cite{Su18}, 
but in both cases, the proof argument is not essentially new. 
In other words, their value lies in the point that the almost same result holds 
even if the condition (K1) in the paper \cite{Su19_1} is generalized as in this paper. 
On the other hand, Theorem \ref{thm_2} 
is new in both assertion and method of proof. 
It further clarifies the relationship between whether $\Theta_E$ is inner 
and the condition (K5) described in \cite{Su19_1}.
\bigskip

For $E \in \mathbb{HB}$,  the condition $E^\sharp(z)=E(-z)$ implies that $\mathcal{H}(E)$ 
is a symmetric de Branges space. 
To handle the non-symmetric cases, we need to consider equations 
$\Phi(t,x)+\int_{-\infty}^{t}K(x+y)\overline{\Phi(t,y)}\,dy=1$ 
and 
$\Psi(t,x)-\int_{-\infty}^{t}K(x+y)\overline{\Psi(t,y)}\,dy=1$
 instead of \eqref{s304_1} and \eqref{s304_2}. 
The results for these cases will be discussed in the upcoming paper in preparation.
\bigskip

The paper is organized as follows. 
In Section \ref{section_2}, we study basic properties of solutions $\Phi(t,x)$ and $\Psi(t,x)$ 
of integral equations \eqref{s304_1} and \eqref{s304_2} in preparation for the proof of Theorem \ref{thm_1},  
and prove Theorem \ref{thm_5}. 
In Section \ref{section_3}, we prove Theorem \ref{thm_1} by using results in Section \ref{section_2}. 
In Section \ref{section_4}, we prove Theorems \ref{thm_2} and \ref{thm_3} by studying the behavior of the operator norm of $\mathsf{K}[t]$ when $t$ varies. 
In Section \ref{section_5}, we review the theory of model subspaces and de Branges spaces in preparation for the proof of Theorem \ref{thm_4}. 
In Section \ref{section_6}, we prove Theorem \ref{thm_4} by studying the vectors representing the point evaluation maps in a model subspace. 
In Section \ref{section_7}, we state and prove several results related to model subspaces. 
In Section \ref{section_8}, we comment on a relation between our inverse problem, 
the Cauchy problem for certain hyperbolic systems, and damped wave equations. 
\bigskip

\noindent
{\bf Acknowledgments}~
This work was supported by JSPS KAKENHI Grant Number JP17K05163.

%
%%%%%%%%%%%%%%%%%%%%%%%%%%%%%%%%%%%%%%%%%%%%%%%%%%%%%%%%%%%%%%%%%%%%%%%%%%%%%%%%%%%%%%%%
%
\section{Solutions of related integral equations} \label{section_2}
%
%%%%%%%%%%%%%%%%%%%%%%%%%%%%%%%%%%%%%%%%%%%%%%%%%%%%%%%%%%%%%%%%%%%%%%%%%%%%%%%%%%%%%%%%
%

We suppose that $E$ satisfies (K2)$\sim$(K5) throughout this section. 
In particular, we understand that $c$ and $\tau$ are constants in (K2) and (K5), respectively.  
\comment{
Let $L^p(I)$ be the $L^p$-space on an interval $I$ with respect to the Lebesgue measure.  
If $J \subset I$, we regard $L^p(J)$ as a subspace of $L^p(I)$ 
by the extension by zero. 
}

%
%%%%%%%%%%%%%%%%%%%%%%%%%%%%%%%%%%%%%%%%%%%%%%%%%%%%%%%%%%%%%%%%%%%%%%%%%%%%%%%%%%%%%%%%
%
\subsection{Properties of $\Phi(t,x)$ and $\Psi(t,x)$} 
%
%%%%%%%%%%%%%%%%%%%%%%%%%%%%%%%%%%%%%%%%%%%%%%%%%%%%%%%%%%%%%%%%%%%%%%%%%%%%%%%%%%%%%%%%
%

\begin{proposition} \label{prop_190620_1}
The integral equations 
\eqref{s304_1} and \eqref{s304_2} 
for $(x,t) \in \R \times (-\infty,\tau)$ have unique solutions $\Phi(t,x)$ and $\Psi(t,x)$, respectively.  
Moreover, 
\begin{enumerate}
\item $\Phi(t,x)$ and $\Psi(t,x)$ are real-valued, 
\item $\Phi(t,x)$ and $\Psi(t,x)$ 
are continuously differentiable functions of $x \in \R$, 
\item $\Phi(t,x)=\Psi(t,x)=1$ for $x<-t$ and $\Phi(t,x) \ll \exp(c'x)$ and $\Psi(t,x) \ll \exp(c'x)$ 
as $x \to +\infty$ for any $t<\tau$ and $c'>c$, where implied constants depend on $t$,  
\item $\Phi(t,t)\not=0$ and $\Psi(t,t)\not=0$ for every $t<\tau$,  
\item if $t \leq 0$, $\Phi(t,t)=\Psi(t,t)=1$ and 
\begin{equation} \label{190706_2}
\Phi(t,x) = 1 - \int_{0}^{x+t} K(y) \, dy, \qquad 
\Psi(t,x) = 1 + \int_{0}^{x+t} K(y) \, dy.
\end{equation}
\end{enumerate}
\end{proposition}
\begin{proof} We prove only in the case of $\Phi(t,x)$, 
because the case of $\Psi(t,x)$ is proved in a similar argument.
First, we prove the uniqueness of $\Phi(t,x)$. 
If $\Phi_1(t,x)$ and $\Phi_2(t,x)$ solve \eqref{s304_1}, 
the difference $f(t,x)= \Phi_1(t,x)-\Phi_2(t,x)$ satisfies 
$f(t,x) + \int_{-\infty}^{t} K(x+y) f(t,y) \, dy = 0$. 
This shows that $f(t,x)=0$ 
if $x \leq t$, since $(1+\mathsf{K}[t])$ is invertible, and hence $f(t,x)=0$. 
Then (1) is obvious, since the kernel $K$ is real-valued by (K2). 
To prove other statements, we first suppose that $t>0$. 

We prove (2) and (3). By considering equation \eqref{s304_1} on $L^2(-t,t)$, 
we find that $\mathbf{1}_{[-t,t]}(x)\Phi(t,x)$ is a continuous function of $x$ on $[-t,t]$ 
by the continuity of $K$ and $\mathbf{1}_{[-t,t]}(x)$, 
where $\mathbf{1}_{A}$ stands for the characteristic function of $A \subset \R$. 
On the other hand, $\Phi(t,x)=1$ for $x <-t$, since the integral in \eqref{s304_1} is zero for $x<-t$ by (K3). 
Therefore, 
\begin{equation} \label{eq_190620_01}
\Phi(t,x) = 1 - \int_{-x}^{-t} K(x+y) \, dy - \int_{-t}^{t} K(x+y) \Phi(t,y) \, dy 
\end{equation}
by \eqref{s304_1}, where the middle integral is understood as zero if $x<t$. 
This equality and (K2) shows that $\Phi(t,x)$ is a continuous function of $x \in \R$ 
and satisfies the estimate in (3). 
Moreover, differentiating \eqref{s304_1} with respect to $x$, 
\begin{equation} \label{eq_180330_04}
\frac{\partial}{\partial x}\Phi(t,x) + \int_{-\infty}^{t} K'(x+y) \Phi(t,y) \, dy =0. 
\end{equation}
This shows that $\Phi(t,x)$ is differentiable for $x$ and 
$(\partial/\partial x)\Phi(t,x)$ is a continuous function of $x$ by (K4). 

We prove (4) by contradiction. Differentiating \eqref{s304_1} with respect to $x$, 
and then applying integration by parts, 
\begin{equation} 
\frac{\partial}{\partial x}\Phi(t,x) + K(x+t)\Phi(t,t) - \int_{-\infty}^{t}K(x+y) \frac{\partial}{\partial y}\Phi(t,y) \, dy =0.  
\end{equation}
Therefore, if we suppose that $\Phi(t,t)=0$, 
\begin{equation} \label{eq_180321_01} 
\frac{\partial}{\partial x}\Phi(t,x) - \int_{-\infty}^{t}K(x+y) \frac{\partial}{\partial y}\Phi(t,y) \, dy =0.  
\end{equation}
This asserts that the restriction $\mathbf{1}_{[-t,t]}(x)(\partial/\partial x)\Phi(t,x)$ 
is a solution of the homogeneous equation $(1-\mathsf{K}[t])f=0$ on $L^2(-t,t)$, 
and thus $(\partial/\partial x)\Phi(t,x)=0$ by \eqref{eq_180321_01} 
and the equality $\Phi(t,x)=1$ for $x<-t$, 
since $K(x+y)=0$ if $x<-t$ and $y<t$. 
Hence, we have 
$c \left(1 + \int_{0}^{x+t}K(y)\,dy \right) = 1$ 
for arbitrary $x$ if $\Phi(t,x)=c$. 
This implies that $K \equiv 0$ on $\R$, and therefore, 
$\Phi(t,x) = 1$ for all $x \in \R$ by \eqref{s304_1}. 
This is a contradiction. 

Finally, we prove (5). If $t \leq 0$, $K(x+y)=0$ for $x < t$ and $y < t$. 
Thus $\Phi(t,x)=1$ for $x < t$, 
the first equality of \eqref{190706_2} 
holds, and $\Phi(t,t)=1$. 
Hence $\Phi(t,x)$ is a continuously differentiable function of $x$ on $\R$ 
satisfying the desired estimate. 
\end{proof}

For convenience of studying the solutions $\Phi(t,x)$ and $\Psi(t,x)$, 
we consider the solutions of integral equations 
\begin{equation} \label{s304_3}
\phi^+(t,x) + \int_{-\infty}^{t} K(x+y) \phi^+(t,y) \, dy = K(x+t), 
\end{equation}
\begin{equation} \label{s304_4}
\phi^-(t,x) - \int_{-\infty}^{t} K(x+y) \phi^-(t,y) \, dy = K(x+t). 
\end{equation}
The usefulness of solutions $\phi^+(t,x)$ and $\phi^-(t,x)$ comes from relationships 
with the resolvent kernels $R^+(t;x,y)$ and $R^-(t;x,y)$ of $(1+\mathsf{K}[t])$ and $(1-\mathsf{K}[t])$, respectively: 
$\mathbf{1}_{\leq t}(x)\phi^+(t,x)=R^+(t;x,t)$, $\mathbf{1}_{\leq t}(x)\phi^-(t,x)=R^-(t;x,t)$ 
(\cite[Section 2]{Su19_1}). 

\begin{proposition} \label{prop_190620_2} 
The integral equations \eqref{s304_3} and \eqref{s304_4} for $(x,t) \in \R \times (-\infty,\tau)$ have unique solutions 
$\phi^+(t,x)$ and $\phi^-(t,x)$, respectively.  
Moreover, 
\begin{enumerate}
\item $\phi^+(t,x)$ and $\phi^-(t,x)$ are continuous on $\R$ 
and continuously differentiable on $\R \setminus \{\lambda -t\,|\,\lambda \in \Lambda\}$ 
as a function of $x$, where $\Lambda$ is the set in (K4). 
\item $\phi^+(t,x)$ and $\phi^-(t,x)$ are continuous on $[0,\tau)$ 
and continuously differentiable on $(0,\tau)$ 
except for points in $\{\lambda - x \,|\,\lambda \in \Lambda\}$ 
as a function of $t$, 
\item for fixed $t \in [0,\tau)$, $\phi^\pm(t,x)=0$ for $x < -t$ and 
$\phi^\pm(t,x) \ll e^{cx}$ as $x \to +\infty$ for $c>0$ in (K2), 
where the implied constants depend on $t$,  
\item if $t \leq 0$, $\phi^+(t,t)=\phi^-(t,t)=0$ and 
\[
\phi^+(t,x) = \phi^-(t,x) = K(x+t), 
\]
\item  the following equations hold 
\begin{equation} \label{s334}
\phi^{\pm}(t,t) = \pm \frac{d}{dt}\log \det(1 \pm \mathsf{K}[t]), 
\end{equation} 
\item the following equation holds 
\begin{equation} \label{s323} 
m(t) = \exp\left( \int_{0}^{t}\mu(s) \, ds \right), 
\qquad 
\mu(t) = \phi^+(t,t) +  \phi^{-}(t,t), 
\end{equation}
\item $m(t)$ is a continuous positive real-valued function on $(-\infty,\tau)$, $m(0)=1$, 
and continuously differentiable outside a discrete subset. 
% by properties of $\mu(t)$, 
\end{enumerate}
\end{proposition}
\begin{proof}
The proof is the same as that in \cite[Section 2]{Su19_1}, 
because conditions (K2)$\sim$(K5) are exactly the same. 
\end{proof}

\begin{proposition} \label{prop_190706_1} 
Let $\Phi(t,x)$ and $\Psi(t,x)$ be unique solutions of 
\eqref{s304_1} and \eqref{s304_2}, respectively. 
Let $\phi^+(t,x)$ and $\phi^-(t,x)$ be unique solutions of 
\eqref{s304_3} and \eqref{s304_4}, respectively.  
Then, $\Phi(t,x)$ and $\Psi(t,x)$ are continuously differentiable with respect to $t$ on $(-\infty,\tau)$ 
and equalities 
\begin{equation} \label{formula_02}
\phi^+(t,x) = - \frac{1}{\Phi(t,t)} \frac{\partial}{\partial t}\Phi(t,x) = \frac{1}{\Psi(t,t)} \frac{\partial}{\partial x}\Psi(t,x), 
\end{equation}
\begin{equation} \label{formula_02_b}
\phi^-(t,x) = - \frac{1}{\Phi(t,t)} \frac{\partial}{\partial x}\Phi(t,x) = \frac{1}{\Psi(t,t)} \frac{\partial}{\partial t}\Psi(t,x)
\end{equation}
hold. 
\end{proposition}
\begin{proof}
Applying integration by parts to \eqref{eq_180330_04}, 
\begin{equation} \label{eq_180321_02}
\frac{\partial}{\partial x}\Phi(t,x) + K(x+t)\Phi(t,t) - \int_{-\infty}^{t}K(x+y) \frac{\partial}{\partial y}\Phi(t,y) \, dy =0.  
\end{equation}
This shows that $-(\partial/\partial x)\Phi(t,x)/\Phi(t,t)$ solves 
\eqref{s304_4} by Proposition \ref{prop_190620_1}\,(4). 
Hence the uniqueness of the solution of \eqref{s304_3} concludes the first equality of \eqref{formula_02_b}. 

On the other hand, we find that 
$\mathbf{1}_{[-t,t]}(x)\Phi(t,x)$ is also continuous in $t$, 
because the resolvent kernel $R^+(t;x,y)$ of $(1+\mathsf{K}[t])$ is continuous 
in all variables (\cite[Section 2]{Su19_1}). 
Therefore, $\Phi(t,x)$ is continuous in $t$ by \eqref{eq_190620_01}. 
By differentiating \eqref{s304_1} with respect to $t$ 
(in the sense of weak derivative), 
we find that $-(\partial/\partial t)\Phi(t,x)/\Phi(t,t)$ 
solves \eqref{s304_3} by Proposition \ref{prop_190620_1}\,(4). 
Hence the uniqueness of the solution concludes the first equality of \eqref{formula_02}. 
Moreover, the first equality of \eqref{formula_02} 
shows that $(\partial/\partial t)\Phi(t,x)$ is continuous with respect to $t$ 
by Proposition \ref{prop_190620_2}\,(2), (3), and (4). 
Hence $\Phi(t,x)$ is differentiable with respect to $t$ in the usual sense, 
and the derivative with respect to $t$ is continuous in $t$. 
The differentiability of $\Psi(t,x)$ with respect to $t$ 
and the second equalities of \eqref{formula_02} and \eqref{formula_02_b} 
are proved by a similar argument. 
\end{proof}

%
%%%%%%%%%%%%%%%%%%%%%%%%%%%%%%%%%%%%%%%%%%%%%%%%%%%%%%%%%%%%%%%%%%%%%%%%%%%%%%%%%%%%%%%%
%
\subsection{Proof of Theorem \ref{thm_5}} 
%
%%%%%%%%%%%%%%%%%%%%%%%%%%%%%%%%%%%%%%%%%%%%%%%%%%%%%%%%%%%%%%%%%%%%%%%%%%%%%%%%%%%%%%%%
%
Taking $x = t$ in equation \eqref{s304_1}  and then differentiating it with respect to $t$, 
\[
\aligned 
0 
= & \, \frac{d}{dt}(\Phi(t,t)) + 2 K(2t)\Phi(t,t) \\
\quad & - \int_{-\infty}^{t}K(t+y)\frac{\partial}{\partial y}\Phi(t,y) \, dy 
+ \int_{-\infty}^{t}K(t+y)\frac{\partial}{\partial t}\Phi(t,y) \, dy. 
\endaligned 
\]
Using the first equalities of \eqref{formula_02} and \eqref{formula_02_b} on the right-hand side, 
\begin{equation} \label{eq_180321_03}
\aligned 
\frac{d}{dt}(\Phi(t,t)) + & \,2 K(2t)\Phi(t,t)  - \Phi(t,t) \int_{-\infty}^{t}K(t+x)(\phi^+(t,x) - \phi^-(t,x)) \, dx =0.
\endaligned 
\end{equation}
On the other hand, by adding both sides of \eqref{s304_3} and \eqref{s304_4}, 
we have
\[
\aligned 
\frac{1}{2}&(\phi^+(t,x) + \phi^-(t,x))   = K(x+t) - \int_{-\infty}^{t}K(x+y)\frac{1}{2}(\phi^+(t,y) - \phi^-(t,y))\,dy.  
\endaligned 
\]
Substituting this into \eqref{eq_180321_03} after taking $x=t$, we get 
\begin{equation} \label{eq_180321_04}
\frac{d}{dt}(\Phi(t,t)) + \Phi(t,t)(\phi^+(t,t) + \phi^-(t,t)))  =0.
\end{equation}
Therefore, 
$
\Phi(t,t) = C \exp\left( -\int_{0}^{t} (\phi^+(\tau,\tau) + \phi^-(\tau,\tau)), d\tau \right) = C m(t)^{-1} 
$
by \eqref{s323}. 
To determine $C$, we take $x=t=0$ in equation \eqref{s304_1}. 
Then $\Phi(0,0)=1$, since the integral on the left-hand side 
is zero because $K(x)=0$ for $x<0$, 
and thus $C=1$ by $m(0)=1$. 
Hence we obtain the first equality of \eqref{formula_01}. 
The second equality of \eqref{formula_01} is proved by the same way. 
\hfill $\Box$
\bigskip

From Theorem \ref{thm_5} and Proposition \ref{prop_190620_1}, 
we find that $H$ of \eqref{s333} is a Hamiltonian on $(-\infty,\tau)$ 
such that it consists of continuous functions 
and has no $H$-indivisible intervals. % that is, all points of $(-\infty,\tau)$ are regular. 
These properties also obtained from Proposition \ref{prop_190620_2}. 

%
%%%%%%%%%%%%%%%%%%%%%%%%%%%%%%%%%%%%%%%%%%%%%%%%%%%%%%%%%%%%%%%%%%%%%%%%%%%%%%%%%%%%%%%%
%
\subsection{Corollaries of Proposition \ref{prop_190706_1} } 
%
%%%%%%%%%%%%%%%%%%%%%%%%%%%%%%%%%%%%%%%%%%%%%%%%%%%%%%%%%%%%%%%%%%%%%%%%%%%%%%%%%%%%%%%%
%

Here we state a few results that easily obtained from Proposition \ref{prop_190706_1}, 
but note that these are of their own interest and are not used to prove the main results. 

\begin{proposition} 
The solutions of \eqref{s304_1} and \eqref{s304_2} are related to each other as follows 
\begin{equation} \label{eq_190620_04}
\Psi(t,x) = 1 - \frac{1}{\Phi(t,t)^2} \int_{-t}^{x}  \frac{\partial}{\partial t}\Phi(t,y) \, dy, 
\end{equation}
\begin{equation} \label{eq_190620_05}
\Phi(t,x) = 1 - \frac{1}{\Psi(t,t)^2}  \int_{-t}^{x} \frac{\partial}{\partial t}\Psi(t,x)\, dy.
\end{equation}
\end{proposition}
\begin{proof} 
Integrating the second equalities of \eqref{formula_02} and using \eqref{formula_01}, 
Proposition \ref{prop_190620_1}\,(3) and Proposition \ref{prop_190620_2}\,(6), 
\begin{equation} \label{190710_1}
\Psi(t,x) = 1 + m(t) \int_{-t}^{x} \phi^+(t,y)\, dy.
\end{equation}
Substitute the first equality of \eqref{formula_02} and \eqref{formula_01} into \eqref{190710_1}, 
we obtain \eqref{eq_190620_04}. \eqref{eq_190620_05} is also proved similarly.
\end{proof}

\begin{proposition} \label{prop_190709_1} We have 
\begin{equation} \label{eq_190620_02} 
\frac{1}{m(t)} = 1 - \int_{-t}^{t} \phi^+(t,y) \, dy, \qquad 
m(t)= 1 + \int_{-t}^{t} \phi^-(t,y) \, dy.
\end{equation}
\end{proposition}
\begin{proof} The first equality is obtained by taking $x=t$ in \eqref{190710_1} 
and noting \eqref{formula_01}.  
The second equality is proved similarly. 
\end{proof}

\begin{proposition}
The following partial differential equations hold
\begin{equation} \label{s324}
\frac{\partial}{\partial t} \phi_t^{+}(x) - \frac{\partial}{\partial x} \phi_t^{-}(x) 
 = - \mu(t) \phi_t^{-}(x), \qquad 
\frac{\partial}{\partial t} \phi_t^{-}(x) - \frac{\partial}{\partial x} \phi_t^{+}(x) 
 = \mu(t) \phi_t^-(x).
\end{equation}
\end{proposition}
\begin{proof} We have 
\[
\frac{\partial}{\partial t}\phi^+(t,x) = -\frac{\Psi_t(t,t)+\Psi_x(t,t)}{\Psi(t,t)}\phi^+(t,x) + \frac{\Psi_{xt}(t,x)}{\Psi(t,t)}, \quad 
\frac{\partial}{\partial x}\phi^-(t,x) = \frac{\Psi_{tx}(t,x)}{\Psi(t,t)}
\]
by \eqref{formula_02} and \eqref{formula_02_b}. On the other hand, 
\[
\frac{\Psi_t(t,t)+\Psi_x(t,t)}{\Psi(t,t)} = \frac{d}{dt} \log m(t) = \mu(t).  
\]
by \eqref{formula_01} and Proposition \ref{prop_190620_2}\,(6). 
Hence we obtain the first equation. The second equation is proved similarly. 
\end{proof}

For convenience in later sections, we put 
\begin{equation} \label{190709_3}
\frak{A}(t,z)=m(t)^{-1}A(t,z), \qquad \frak{B}(t,z)=m(t)B(t,z). 
\end{equation}

\begin{lemma} \label{lem_5_1} For $\Im(z)>c$, we have 
\begin{equation} \label{190703_1}
\aligned
\frak{A}(t,z)
& = \frac{1}{2}E(z)\left( e^{izt} + \int_{t}^{\infty} \phi^+(t,x) e^{izx} \, dx \right), 
\\ 
-i \frak{B}(t,z)
& = \frac{1}{2}E(z)\left( e^{izt} - \int_{t}^{\infty} \phi^-(t,x) e^{izx} \, dx \right). 
\endaligned
\end{equation}
%
% Here (K5) is necessary but $\tau=\infty$ is unnecessary. 
%
\end{lemma}
\begin{proof}
Using \eqref{formula_02} and \eqref{formula_02_b}, 
\[
e^{izt} + \int_{t}^{\infty} \phi^+(t,x) e^{izx} \, dx 
=  - iz \int_{t}^{\infty} \frac{\Psi(t,x)}{\Psi(t,t)} e^{izx} \, dx,
\]
\[
e^{izt} - \int_{t}^{\infty} \phi^-(t,x) e^{izx} \, dx 
= - iz \int_{t}^{\infty}  \frac{\Phi(t,x)}{\Phi(t,t)} e^{izx} \, dx.
\]
Hence we obtain \eqref{190703_1} by definition \eqref{eq_190625_1}. 
The convergence of integrals is justified by Proposition \ref{prop_190620_2}\,(3). 
\end{proof}

%
%%%%%%%%%%%%%%%%%%%%%%%%%%%%%%%%%%%%%%%%%%%%%%%%%%%%%%%%%%%%%%%%%%%%%%%%%%%%%%%%%%%%%%%%
%
\section{Proof of Theorem \ref{thm_1}} \label{section_3}
%
%%%%%%%%%%%%%%%%%%%%%%%%%%%%%%%%%%%%%%%%%%%%%%%%%%%%%%%%%%%%%%%%%%%%%%%%%%%%%%%%%%%%%%%%
%

We suppose that $E$ satisfies (K1)$\sim$(K5) throughout this section. 
%that is, we suppose (K1) further. 
%
We use the same notation \eqref{eq_Fourier} for the Fourier transforms on $L^1(\R)$ and $L^2(\R)$ if no confusion arises. 
If we understand the right-hand sides of \eqref{eq_Fourier}  in $L^2$-sense, 
they provide isometries on $L^2(\R)$ up to a constant multiple: 
$\Vert \mathsf{F}f \Vert_{L^2(\R)}^2 = 2\pi \Vert f \Vert_{L^2(\R)}^2$, 
$\Vert \mathsf{F}^{-1}f \Vert_{L^2(\R)}^2 = (2\pi)^{-1} \Vert f \Vert_{L^2(\R)}^2$. 

%
%%%%%%%%%%%%%%%%%%%%%%%%%%%%%%%%%%%%%%%%%%%%%%%%%%%%%%%%%%%%%%%%%%%%%%%%%%%%%%%%%%%%%%%
%
\subsection{Proof of (1)} 
%
%%%%%%%%%%%%%%%%%%%%%%%%%%%%%%%%%%%%%%%%%%%%%%%%%%%%%%%%%%%%%%%%%%%%%%%%%%%%%%%%%%%%%%%
%
From \eqref{s304_2}, we have 
\[
\aligned 
\mathbf{1}_{[t,\infty)}(x)\Psi(t,x) 
& =\mathbf{1}_{(-\infty,-t)}(x) - \mathbf{1}_{(-\infty,t)}(x)\Psi(t,x) \\
& \quad +  \mathbf{1}_{[-t,\infty)}(x) +\int_{-\infty}^{t}K(x+y)\Psi(t,y)\,dy.
\endaligned 
\]
By taking Fourier transform of both sides for $\Im z>c$  
using Proposition \ref{prop_190620_1} (3), 
\[
\aligned 
\int_{t}^{\infty} \Psi(t,x) e^{izx} \, dx  
& = - \int_{-t}^{t} \Psi(t,x) e^{izx} \, dx \\
& \quad + \frac{e^{-izt}}{-iz}+\frac{E(-z)}{E(z)}
\left\{
\int_{-t}^{t}\Psi(t,x)e^{-izx}\,dx 
+ 
\frac{e^{izt}}{-iz}
\right\}
\endaligned 
\]
Since the right-hand side is analytic in $\C\setminus\mathcal{Z}$, 
it gives the analytic continuation 
of the left-hand side to this region. 
Also, by definition \eqref{eq_190625_1}, 
\[
\aligned 
A(t,z) 
& =  \frac{E(z)}{2}\left\{ e^{-izt}  + iz  \int_{-t}^{t} \Psi(t,x) e^{izx} \, dx \right\} \\
& \quad 
+\frac{E(-z)}{2}
\left\{ e^{izt} + i(-z)\int_{-t}^{t}\Psi(t,x)e^{-izx}\,dx
\right\}, \\
-iB(t,z)
& =  \frac{E(z)}{2}\left\{ e^{-izt}  + iz  \int_{-t}^{t} \Phi(t,x) e^{izx} \, dx \right\} \\
& \quad 
-\frac{E(-z)}{2}
\left\{ e^{izt} + i(-z)\int_{-t}^{t}\Phi(t,x)e^{-izx}\,dx 
\right\}.
\endaligned 
\]
Since $\Phi(t,x)$ and $\Psi(t,x)$ are obviously real-valued functions, 
these formulas and  (K1) show the desired properties of $A(t,z)$ and $B(t,z)$. 
\hfill $\Box$ 

%
%%%%%%%%%%%%%%%%%%%%%%%%%%%%%%%%%%%%%%%%%%%%%%%%%%%%%%%%%%%%%%%%%%%%%%%%%%%%%%%%%%%%%%%%
%
\subsection{Proof of (2) and (3)} 
%
%%%%%%%%%%%%%%%%%%%%%%%%%%%%%%%%%%%%%%%%%%%%%%%%%%%%%%%%%%%%%%%%%%%%%%%%%%%%%%%%%%%%%%%%
%
By Proposition \ref{prop_190706_1}, 
formulas \eqref{formula_02}, \eqref{formula_02_b},  
and Proposition \ref{prop_190620_2}\,(3), 
$A(t,z)$ and $B(t,z)$ are differentiable with respect to $t$. 
Therefore, it remains to show that 
$(\partial/\partial t)A(t,z)$ $= z m(t)^2B(t,z)$ 
and $(\partial/\partial t)B(t,z)= -z m(t)^{-2}A(t,z)$.  
Using  \eqref{formula_01} and \eqref{formula_02_b}, we have 
\[
\aligned 
\frac{\partial}{\partial t} A(t,z)
& = - \frac{iz}{2}E(z) \left( 
- \Psi(t,t) e^{izt} + \int_{t}^{\infty} \frac{\partial}{\partial t}\Psi(t,x) e^{izx} \, dx
\right) \\
& = - \frac{iz}{2}E(z) \left( 
- \Psi(t,t) e^{izt} - \frac{\Psi(t,t)}{\Phi(t,t)} \int_{t}^{\infty} \frac{\partial}{\partial x}\Phi(t,x) e^{izx} \, dx
\right) \\
& = - \frac{iz}{2}E(z) \left( 
 \frac{\Psi(t,t)}{\Phi(t,t)} iz\int_{t}^{\infty} \Phi(t,x) e^{izx} \, dx
\right)
= z \frac{\Psi(t,t)}{\Phi(t,t)} B(t,z) = z m(t)^2 B(t,z). 
\endaligned 
\]
The second equality is proved similarly. 
\hfill $\Box$

%
%%%%%%%%%%%%%%%%%%%%%%%%%%%%%%%%%%%%%%%%%%%%%%%%%%%%%%%%%%%%%%%%%%%%%%%%%%%%%%%%%%%%%%%%
%
\subsection{Proof of (4)} 
%
%%%%%%%%%%%%%%%%%%%%%%%%%%%%%%%%%%%%%%%%%%%%%%%%%%%%%%%%%%%%%%%%%%%%%%%%%%%%%%%%%%%%%%%%
%
For $t \leq 0$, we have
\begin{equation} \label{s331}
\aligned 
A(t,z) & = A(z)\cos(tz)+B(z)\sin(tz), \\
B(t,z) & = -A(z)\sin(tz)+B(z)\cos(tz) \\
\endaligned 
\end{equation} 
by \eqref{190706_2}, \eqref{s104} and \eqref{eq_190625_1}. 
In particular, Theorem \ref{thm_1} (4) holds. 
%
%%%%%%%%%%%%%
\comment{
Integrating the second equality of \eqref{190706_2} 
on $x$ from $t$ to $\infty$, 
\[
\aligned 
A(t,z)
& = -\frac{iz}{2}E(z)\left( -\frac{1}{iz}e^{izt} -\frac{1}{iz}e^{-izt}\int_{0}^{\infty} K(y) e^{izy}\,dy \right) \\
& = \frac{1}{2}E(z)\left(e^{itz}+e^{-itz}\Theta(z) \right) 
= A(z)\cos(tz)+B(z)\sin(tz)
\endaligned 
\] 
if $\Im(z)>c$ by \eqref{s104} and \eqref{eq_190625_1}. 
The case of $B(t,z)$ is proved in the same way. 
}
%%%%%%%%%%%%%
\hfill $\Box$

%
%%%%%%%%%%%%%%%%%%%%%%%%%%%%%%%%%%%%%%%%%%%%%%%%%%%%%%%%%%%%%%%%%%%%%%%%%%%%%%%%%%%%%%%%
%
\section{Proof of Theorems \ref{thm_2} and \ref{thm_3}} \label{section_4}
%
%%%%%%%%%%%%%%%%%%%%%%%%%%%%%%%%%%%%%%%%%%%%%%%%%%%%%%%%%%%%%%%%%%%%%%%%%%%%%%%%%%%%%%%%
%

\begin{lemma} \label{lem_3_2}
Let $E$ be a function satisfying (K1)$\sim$(K3). 
Define $\mathsf{K}f=\lim_{t\to\infty}\mathsf{K}[t]f$ 
in pointwise convergence for $f$ in the space $C_c^{\infty}(\R)$ of all compactly supported smooth function on $\R$. 
Then the Fourier integral formula 
\begin{equation} \label{s303}
(\mathsf{F}\mathsf{K}f)(z) = \Theta_E(z)\,(\mathsf{F}f)(-z)
\end{equation}
holds for $\Im(z)>c$. 
Suppose that $\Theta_E=E^\sharp/E$ is an inner function in $\C_+$ in addition to (K1)$\sim$(K3). 
Then $\mathsf{K}f$ belongs to $L^2(\R)$ for $f \in C_c^{\infty}(\R)$, 
and the linear map $f \mapsto \mathsf{K} f$ extends to the isometry 
$\mathsf{K}: L^2(\R) \to L^2(\R)$ satisfying $\mathsf{K}^2={\rm id}$ and 
\eqref{s303} holds for $z \in \R\setminus\mathcal{Z}$. 
Moreover, \eqref{s303} holds for $z \in (\C_+ \cup \R)\setminus\mathcal{Z}$, 
if $f \in L^2(\R)$ has support in $(-\infty,t]$ for some $t \in \R$. 
\end{lemma}
\begin{proof} 
This is proved by almost the same argument as \cite[Lemma 2.2]{Su19_1}, 
because the difference of condition (K1) between this paper and \cite{Su19_1} 
does not affect the argument of the proof. 
\end{proof}

%%%%%%%%%%%%
\comment{
\begin{corollary} 
Let $E$ be a function satisfying (K1)$\sim$(K3). Suppose that $\Theta=E^\sharp/E$ is inner in $\C_+$. 
Then the integral $\int_{0}^{\infty} K(x)e^{izx} \, dx$ absolutely converges 
and $\Theta(z)=\mathsf{F}K(z)$ for $z \in \C_+$. 
\end{corollary}
\begin{proof}
For $z \in \C_+$, $\mathsf{1}_{\geq -x}(y)e^{izy}$ belongs to $L^2(0,\infty)$. 
Thus $\int_{-x}^{\infty} K(x+y)e^{izy} \, dy = e^{-izx}\int_{0}^{\infty} K(y)e^{izy} \, dy$ 
{\color{red} あれ? $H^2 \subset L^2$ の Fourier 変換の一般論でいけないか?}
\end{proof}
}
%%%%%%%%%%%%%

Here we recall the basic properties of the Hardy spaces. 
The Hardy space $H^2 = H^2(\C_+)$ in $\C_+$ 
is defined to be the space of all analytic functions $f$ in $\C_+$ endowed with 
norm $\Vert f \Vert_{H^2}^2 := \sup_{v>0} \int_{\R} |f(u+iv)|^2 \, du < \infty$. 
The Hardy space $\bar{H}^2 = H^2(\C_-)$ in the lower half-plane $\C_-$ is defined similarly. 
As usual, we identify $H^2$ and $\bar{H}^2$ with subspaces of $L^2(\R)$ 
via nontangential boundary values on the real line such that $L^2(\R)=H^2 \oplus \bar{H}^2$, 
where $L^2(\R)$ has the inner product $\langle f,g \rangle = \int_{\R} f(u)\overline{g(u)}du$. 
The Hardy space 
$H^2$ is a reproducing kernel Hilbert space, 
in particular, the point evaluation functional 
$f \mapsto f(z)$ is continuous for each $z \in \C_+$ 
and it is represented by $k_z(w)=(2\pi i)^{-1}(\bar{z}-w)^{-1} \in H^2$%=1/(w+\bar{z}) \in H^2$ 
as $\langle f, k_z \rangle = f(z)$. 

\begin{lemma} \label{lem_190628_1} Let $\theta$ be an analytic function defined in $\C_+$. 
Suppose that $\theta f \in H^2$ for every $f \in H^2$. 
Then the pointwise multiplication operator $M_\theta: f \mapsto \theta f$ 
is bounded on $H^2$ and $\theta \in H^\infty$. 
\end{lemma} 
\begin{proof} See \cite[Lemma 5.1]{Su19_1}. 
\comment{
We find that $M_\theta$ is bounded from the closed graph theorem and continuity of point evaluations. 
Let $k_z \in H^2$ be the vector representing the point evaluation at $z \in \C_+$.
%: $f(z) = \langle f, k_z \rangle$ for all $f \in H^2$. 
Then, we have
\[
%\langle f, M_\Theta^\ast K_z \rangle = 
\langle M_\theta f,  k_z \rangle 
= \langle \theta f,  k_z \rangle 
= \theta(z)f(z)
= \langle f,  \overline{\theta}(z) k_z \rangle. 
\]
for $f \in H^2$. 
Therefore, the adjoint $M_\theta^\ast$ acts on $k_z$ as $M_\theta^\ast k_z = \overline{\theta(z)}k_z$. 
This implies $k_z$ is an eigenvector of $M_\theta^\ast$ with eigenvalue $\overline{\theta(z)}$. 
Hence, $|\theta(z)| \leq \Vert M_\theta \Vert_{\rm op}$ for every $z \in \C_+$. 
This shows that $\theta(z)$ is uniformly bounded on $\C_+$. 
%since the boundedness of $M_\Theta$ implies the boundedness of $M_\Theta^\ast$. 
}
\end{proof}

\begin{proposition} \label{prop_190626_2}
Suppose that $E$ satisfies (K1)$\sim$(K3). 
Then the following are equivalent:
\begin{enumerate}
\item $\Theta_E=E^\sharp/E$ is an inner function in $\C_+$. 
\item $\mathsf{K}[t]$ converges as $t \to \infty$ with respect to the operator norm $\Vert\cdot\Vert_{\rm op}$. 
\item $\mathsf{K}[t]f$ converges as $t \to \infty$ with respect to $\Vert\cdot\Vert_{L^2(\R)}$ for all $f \in L^2(\R)$.
\item there exists $M>0$ such that $\Vert \mathsf{K}[t] \Vert_{\rm op} \leq M$ for every $t>0$. 
%\item $\det(1 \pm \mathsf{K}_t)\not=0$ for every $t>0$.
\end{enumerate}
\end{proposition} 
\begin{proof} 
We write $\Theta_E$ as $\Theta$ for simplicity. 
We prove the implication (2)$\Rightarrow $(1). Suppose that $\mathsf{K}=\lim_{t\to\infty}\mathsf{K}[t]$ exists with respect to the operator norm. 
Then $\mathsf{K}f \in L^2(\R)$ for any $f \in L^2(\R)$, where we understand as  
\[
\mathsf{K}f(x) 
= \int_{-\infty}^{\infty}K(x+y)f(y) \, dy
:= \underset{t \to \infty}{\rm l.i.m}~\mathbf{1}_{(-\infty,t)}(x) \int_{-\infty}^{t}K(x+y)f(y) \, dy, 
\] 
where ${\rm l.i.m}$  stands for limit in mean. 
We have $\mathsf{K}[t]=\mathsf{P}_t\mathsf{K}\mathsf{P}_t$. 
Let $f \in L^2(-\infty,0)$ and take $\{f_n\}_n \subset C_c^\infty(\R)$ such that $\lim_{n\to\infty}f_n=f$ in $L^2(\R)$. 
Then, 
\[
(\mathsf{F}\mathsf{K}f)(z) 
= \lim_{n \to \infty}(\mathsf{F}\mathsf{K}f_n)(z) 
= \lim_{n \to \infty}\lim_{t \to \infty} (\mathsf{F}\mathsf{K}[t]f_n)(z), 
\]
since $\mathsf{K}$ and $\mathsf{F}$ are bounded. On the right-hand side, 
\[
\lim_{t \to \infty} (\mathsf{F}\mathsf{K}[t]f_n)(z)
= \Theta(z) (\mathsf{F}f_n)(-z)
\]
for $\Im(z)>c$ by Lemma \ref{lem_3_2}. Therefore, 
$(\mathsf{F}\mathsf{K}f)(z) = \Theta(z) (\mathsf{F}f)(-z)$ for $\Im(z)>c$. 
On the other hand, $\mathsf{K}f \in L^2(0,\infty)$ by (K3). 
Thus $(\mathsf{F}\mathsf{K}f)(z)$ defines an analytic function in $\C_+$. 
Hence $(\mathsf{F}\mathsf{K}f)(z)=\Theta(z) (\mathsf{F}f)(-z)$ 
which conclude $\Theta H^2 \subset H^2$. 
Hence $\Theta \in H^\infty$ by Lemma \ref{lem_190628_1}. 
Therefore, by \eqref{105}, 
$\Theta$ is an inner function in $\C_+$. 

We prove the implication (1)$\Rightarrow $(3). 
If $\Theta$ is an inner function in $\C_+$, for any $f \in L^2(\R)$, 
$\mathsf{L}f:=\mathsf{F}^{-1}M_\Theta\mathsf{F}\mathsf{J}f$ is defined and 
belongs to $L^2(\R)$, 
$\mathsf{L}f(x)=\int K(x+y)f(y)\,dy:={\rm l.i.m}_{T \to \infty} \int_{-T}^{T}K(x+y)f(y) \, dy$,   
and $\Vert \mathsf{L}f\Vert= \Vert f \Vert$ by the argument similar 
to the proof of \cite[Lemma 2.2]{Su19_1}, 
where $\mathsf{J}f(x)=f(-x)$. 
We have $\mathsf{K}[t]=\mathsf{P}_t\mathsf{L}\mathsf{P}_t$. Therefore, 
\[
\mathsf{L}f - \mathsf{K}[t]f = \mathsf{L}(1-\mathsf{P}_t)f + (1-\mathsf{P}_t)\mathsf{L}\mathsf{P}_tf 
\to 0 \quad (t\to\infty). 
\]
Hence $\mathsf{K}[t]f$ converges to $\mathsf{L}f$, and $\mathsf{L}=\mathsf{K}$.  

The implication (3)$\Rightarrow $(2) is a consequence of the Banach--Steinhaus theorem. 
The implication (2)$\Rightarrow$(4) is trivial. 
Finally, we show that (4) implies (3). 
Let $t>s>0$ and $f \in L^2(\R)$, 
and let $\mathsf{P}_t$ be the projection from $L^2(\R)$ to $L^2(-\infty,t)$. 
Then, 
$\mathsf{K}[t]f-\mathsf{K}[s]f 
= \mathsf{P}_t\mathsf{K}(\mathsf{P}_t - \mathsf{P}_s)f 
- (\mathsf{P}_t - \mathsf{P}_s)\mathsf{K}(\mathsf{P}_t-\mathsf{P}_s) f
+ (\mathsf{P}_t - \mathsf{P}_s)\mathsf{K}\mathsf{P}_t f
$,  
and thus 
\[
\aligned 
\Vert \mathsf{K}[t]f-\mathsf{K}[s]f \Vert 
& \leq \Vert \mathsf{K}[t](\mathsf{P}_t - \mathsf{P}_s)f \Vert 
+  \Vert (\mathsf{P}_t - \mathsf{P}_s)\mathsf{K}[t](\mathsf{P}_t-\mathsf{P}_s) f \Vert
+  \Vert (\mathsf{P}_t - \mathsf{P}_s)\mathsf{K}\mathsf{P}_t f \Vert \\
& \leq 2M \Vert (\mathsf{P}_t - \mathsf{P}_s )f \Vert +  \Vert (\mathsf{P}_t - \mathsf{P}_s)\mathsf{K}\mathsf{P}_t f \Vert. 
\endaligned 
\]
The first term of the right-hand side is smaller as $t>s>0$ are larger. 
We show that the second term is also smaller as $t>s>0$ are larger by contradiction.  
Suppose that there exists $a>0$ such that 
\[
a \leq \Vert (\mathsf{P}_t - \mathsf{P}_s)\mathsf{K}\mathsf{P}_t f \Vert^2 
= \int_{s}^{t} \left| \int_{-\infty}^{t}K(x+y)f(y) \,dy \right|^2 dx 
\]
holds for some $t>s>s_0$ for arbitrary $s_0$. 
Then we can take a strictly increasing numbers $s_0<s_1<\dots$ such that 
$ a \leq \int_{s_n}^{s_{n+1}} \left| \int_{-\infty}^{s_{n+1}}K(x+y)f(y) \,dy \right|^2 dx$ 
holds. For such numbers, 
\[
\aligned 
a 
&\leq \int_{s_n}^{s_{n+1}} \left| \int_{-\infty}^{s_{n+1}}K(x+y)f(y) \,dy \right|^2 dx  \\
&\leq \int_{s_n}^{s_{n+1}} \left| \int_{-\infty}^{s_0}K(x+y)f(y) \,dy \right|^2 dx 
+ \int_{s_n}^{s_{n+1}} \left| \int_{s_0}^{s_{n+1}}K(x+y)f(y) \,dy \right|^2 dx \\
&\leq \int_{s_n}^{s_{n+1}} \left| \int_{-\infty}^{s_0}K(x+y)f(y) \,dy \right|^2 dx 
+ \Vert (\mathsf{P}_{s_{n+1}}-\mathsf{P}_{s_n})\mathsf{K}[s_{n+1}](\mathsf{P}_{s_{n+1}}-\mathsf{P}_{s_0})f \Vert^2 \\
&\leq \int_{s_n}^{s_{n+1}} \left| \int_{-\infty}^{s_0}K(x+y)f(y) \,dy \right|^2 dx 
+ M \Vert (\mathsf{P}_{s_{n+1}}-\mathsf{P}_{s_0})f \Vert^2.   
\endaligned 
\]
The second term of the right-hand side is small for large $s_0$. 
Therefore, by replacing $a$ if necessary, 
\[
a \leq \int_{s_n}^{s_{n+1}} \left| \int_{-\infty}^{s_0}K(x+y)f(y) \,dy \right|^2 dx 
\quad (n=0,1,2,\dots). 
\]
This implies $\Vert \mathsf{K}[s_{n+1}](P_{s_0}f) \Vert \geq \Vert \mathsf{K}[s_1](P_{s_0}f) \Vert +na$, 
which contradicts the uniform boundedness (4).  
%
%%%%%%%%%%%%%%%%%%%%%%%%%
\comment{
Suppose that $\Theta$ is inner. 
Then, for any $f \in C_c^\infty(\R)$, $\int_{-\infty}^{t}K(x+y)f(y) \, dy$ is defined and 
\[
\mathsf{F}\left(\int_{-\infty}^{t}K(x+y)f(y) \, dy \right)(z)
=\Theta(z) \mathsf{F}f(-z)
\]
if $\Im(z)>c$ and $t>0$ is large. Therefore, 
\[
\int_{-\infty}^{t}K(x+y)f(y) \, dy = \frac{1}{2\pi}\int_{(c)} \Theta(z) \mathsf{F}f(-z) e^{-izx} \, dz
\]
On the the right-hand side, we can move the path of integration to the real line. 
Hence $\int_{-\infty}^{t}K(x+y)f(y) \, dy$ belongs to $L^2(\R)$, 
because the integrand of the right-hand side is $L^2$ on the real line.
Moreover, $\Vert \mathsf{K}f \Vert_{L^2} = (2\pi)\Vert \Theta \mathsf{F}f \Vert_{L^2}= (2\pi)\Vert \mathsf{F}f \Vert_{L^2} = \Vert f \Vert_{L^2}$. 
}
%%%%%%%%%%%%%%%%%%%%%%%%%
\end{proof}

\begin{proposition} \label{prop_190626_1} We can arrange the eigenvalues of the family $\{\mathsf{K}[t]\}_{t}$ 
such that 
$\lambda_1^+(t) \geq \lambda_2^+(t) \geq \dots$ are positive eigenvalues, 
$\lambda_1^-(t) \leq \lambda_2^-(t) \leq \dots$ are negative eigenvalues, 
where eigenvalues are repeated as many times as multiplicities, 
and each $|\lambda_i^\pm(t)|$ is a continuous nondecreasing function of $t$. 
\end{proposition}
\begin{proof}
We find that $\lambda_j^+(t)$ is a nondecreasing function of $t$ 
by the min--max principle for positive eigenvalues 
$\displaystyle{\lambda_j^+(t) = \min_{V_j}\max_{f \in V_j^\perp}\{ \langle \mathsf{K}[t]f,f\rangle/\langle f,f\rangle\}}$, 
where $V_j$ runs all $(j-1)$-dimensional subspace of $L^2(-\infty,t)$ 
and $V_j^\perp$ stands for the orthogonal complement of $V_j$, 
because the maximum part of the formula is a nondecreasing function of $t$ 
by the inclusion $L^2(-\infty,s)\subset L^2(-\infty,t)$ for $s<t$ 
obtained by the extension by zero. 
Also, $\lambda_j^-(t)$ is a nonincreasing function of $t$ 
by applying the above argument to $-\mathsf{K}[t]$. 
Since $\mathsf{K}[t]$ is a Hilbert--Schmidt operator, 
$\sum_{j}\lambda_j^+(t)^2 + \sum_{j}\lambda_j^-(t)^2
 = \int_{-t}^{t}\int_{-t}^{t}|K(x+y)|^2dxdy$. 
 Therefore, for sufficiently small $h>0$, 
\[
\aligned 
\sum_{j}&(\lambda_j^+(t+h)^2 - \lambda_j^+(t)^2) + \sum_{j}(\lambda_j^-(t+h)^2 - \lambda_j^-(t)^2) \\
&= 
\left(\int_{t}^{t+h}\int_{-t-h}^{t+h} 
+ \int_{-t-h}^{-t}\int_{-t-h}^{t+h}
+ \int_{t}^{t+h}\int_{-t}^{t}
+ \int_{-t-h}^{-t}\int_{-t}^{t}
\right) |K(x+y)|^2dxdy \\
& \leq 4h \int_{-3t}^{3t}|K(x)|^2 \,dx, 
\endaligned 
\]
and hence 
$\lambda_j^\pm(t+h)^2 - \lambda_j^\pm(t)^2  \leq 4h \int_{-3t}^{3t}|K(x)|^2 \,dx$. 
This implies that each $\lambda_j^\pm(t)$ is right continuous. 
The left continuity of each $\lambda_j^\pm(t)$ is proved by the same argument.
\end{proof}
\medskip

\noindent
{\bf Proof of Theorems \ref{thm_2} and \ref{thm_3}.} 
Recall that the operator norm of a compact self-adjoint operator 
is equal to the maximum of absolute values of eigenvalues. 
Then, condition (K5) with $\tau=\infty$ implies that 
$\Vert \mathsf{K}[t] \Vert_{\rm op} <1$ 
by Proposition \ref{prop_190626_1}, 
Hence $\Theta_E=E^\sharp/E$ 
is an inner function in $\C_+$ by Proposition \ref{prop_190626_2}, and Theorem \ref{thm_2} is proved. 
If we suppose that $E$ satisfies (K1)$\sim$(K3) and (K6) and that $\Theta_E$ is an inner function in $\C_+$, 
both $\pm 1$ are not eigenvalues of $\mathsf{K}[t]$ for every $t>0$ 
as proved in \cite[Theorem 5.2]{Su19_1}.  
Then, $\Vert \mathsf{K}[t] \Vert_{\rm op} <1$ by Proposition \ref{prop_190626_1}. 
Hence $\Theta_E$ is an inner function in $\C_+$ by Proposition \ref{prop_190626_2}, 
and Theorem \ref{thm_3} is proved. 
\hfill $\Box$

%
%%%%%%%%%%%%%%%%%%%%%%%%%%%%%%%%%%%%%%%%%%%%%%%%%%%%%%%%%%%%%%%%%%%%%%%%%%%%%%%%%%%%%%%%
%
\section{Theory of model subspaces} \label{section_5}
%
%%%%%%%%%%%%%%%%%%%%%%%%%%%%%%%%%%%%%%%%%%%%%%%%%%%%%%%%%%%%%%%%%%%%%%%%%%%%%%%%%%%%%%%%
%

In this section, 
we review basic notions and properties of model subspaces and de Branges spaces 
in preparation for the next section 
according to Havin--Mashreghi~\cite{MR2016246, MR2016247}, 
Baranov~\cite{MR1855436} for model subspaces 
and de Branges \cite{MR0229011}, Romanov \cite{Rom14}, Winkler \cite{Win14}, Woracek \cite{Wo} for de Branges spaces.

%
%%%%%%%%%%%%%%%%%%%%%%%%%%%%%%%%%%%%%%%%%%%%%%%%%%%%%%%%%%%%%%%%%%%%%%%%%%%%%%%%%%%%%%%%
%
\subsection{Model subspaces and de Branges spaces} 
%
%%%%%%%%%%%%%%%%%%%%%%%%%%%%%%%%%%%%%%%%%%%%%%%%%%%%%%%%%%%%%%%%%%%%%%%%%%%%%%%%%%%%%%%
%

For an inner function $\theta$, 
a {\it model subspace} (or coinvariant subspace) $\mathcal{K}(\theta)$ is defined by the orthogonal complement $\mathcal{K}(\theta)=H^2 \ominus \theta H^2$,
where $\theta H^2 = \{ \theta(z)F(z) \, |\, F \in H^2\}$. 
It has the alternative representation 
$\mathcal{K}(\theta) = H^2 \cap \theta \bar{H}^2$.
A model subspace $\mathcal{K}(\theta)$ is a reproducing kernel Hilbert space 
with respect to the norm induced from $H^2$. 
The reproducing kernel of $\mathcal{K}(\theta)$ is 
\[
j(z,w) = \frac{1-\overline{\theta(z)}\theta(w)}{2\pi i(\bar{z}-w)}, %  \quad (z,w \in \C_+),  
\]
that is, $\langle f, j(z,\cdot) \rangle_{H^2} = f(z)$ for $f \in \mathcal{K}(\theta)$ and $z \in \C_+$. 
For $E \in \overline{\mathbb{HB}}$, the set 
\begin{equation*}
\mathcal{H}(E) := \{ f ~|~ \text{$f$ is entire, $f/E$ and $f^\sharp/E \in  H^2$} \}
\end{equation*}
forms a Hilbert space under the norm $\Vert f \Vert_{\mathcal{H}(E)} := \Vert f/E \Vert_{H^2}$. 
The Hilbert space $\mathcal{H}(E)$ is called the {\it de Branges space} generated by $E$. 
The de Branges space $\mathcal{H}(E)$ is a reproducing kernel Hilbert space 
consisting of entire functions endowed with the reproducing kernel 
\begin{equation*} %\label{rp_1}
J(z,w) = \frac{\overline{E(z)}E(w) - \overline{E^\sharp(z)}E^\sharp(w)}{2\pi i(\bar{z} - w)}.  
%\quad (z,w \in \C_+).  
\end{equation*}
The reproducing formula $\langle f, J(z,\cdot)\rangle_{\mathcal{H}(E)}=f(z)$ for $f \in \mathcal{H}(E)$ and $z \in \C_+$ 
remains true for $z \in  \R$ if $\theta=E^\sharp/E$ is analytic in a neighborhood of $z$. 
\medskip

If an inner function $\theta$ in $\C_+$ is meromorphic in $\C$, 
it is called a {\it meromorphic inner function}. 
It is known that every meromorphic inner function is expressed as 
$\theta = E^\sharp/E$ by using some $E \in \mathbb{HB}$. 
If $\theta$ is a meromorphic inner function such that $\theta=E^\sharp/E$, 
the model subspace $\mathcal{K}(\theta)$ is isomorphic and isometric to the de Branges space $\mathcal{H}(E)$ as a Hilbert space 
by $\mathcal{K}(\theta) \to \mathcal{H}(E): \, f(z) \mapsto E(z)f(z)$. 
\medskip

As developed in \cite{Burnol2002, Burnol2004, Burnol2007, Burnol2011}, 
the Hankel type operator $\mathsf{K}$ with the kernel $K(x+y)$ is useful 
to study model subspaces $\mathcal{K}(\theta)$ or de Branges spaces $\mathcal{H}(E)$ via Fourier analysis. 

%
%%%%%%%%%%%%%%%%%%%%%%%%%%%%%%%%%%%%%%%%%%%%%%%%%%%%%%%%%%%%%%%%%%%%%%%%%%%%%%%%%%%%%%%%
%
\section{Proof of Theorem \ref{thm_4}} \label{section_6}
%
%%%%%%%%%%%%%%%%%%%%%%%%%%%%%%%%%%%%%%%%%%%%%%%%%%%%%%%%%%%%%%%%%%%%%%%%%%%%%%%%%%%%%%%%
%

We suppose that $E$ satisfies (K1)$\sim$(K5) with $\tau=\infty$ throughout this section. %, if not it is mentioned. 
Then $\Theta_E=E^\sharp/E$ is an inner function in $\C_+$ 
by Theorem \ref{thm_2} (thus $\mathcal{Z} \subset \R$) 
and $f \mapsto \mathsf{K}f$ defines an isometry of $L^2(\R)$ satisfying $\mathsf{K}^2={\rm id}$ 
by Lemma \ref{lem_3_2}. 
These properties are essential in the argument of the section. 
Let $\mathcal{V}_t$ be the Hilbert space of all functions $f$ 
such that both $f$ and $\mathsf{K}f$ are square-integrable functions having supports in $[t,\infty)$: 
\begin{equation*}
\mathcal{V}_t = L^2(t,\infty) \cap \mathsf{K}(L^2(t,\infty)). 
\end{equation*}
If $t \geq 0$, $\mathsf{F}(\mathcal{V}_t)$ is a closed subspace of $H^2$, because 
the Fourier transform provides an isometry of $L^2(\R)$ up to a constant 
such that $H^2 = {\mathsf F}L^2(0,\infty)$ and $\bar{H}^2 = {\mathsf F}L^2(-\infty,0)$ by the Paley-Wiener theorem, 
and $\mathcal{V}_t$ is a closed subspace of $L^2(0,\infty)$ by definition. 
Therefore, $\mathsf{F}(\mathcal{V}_t)$ is a reproducing kernel Hilbert space, 
since the point evaluation map $F \mapsto F(z)$ is continuous in $\mathsf{F}(\mathcal{V}_t)$ for each $z \in \C_+$ as well as $H^2$. 

The following Lemmas \ref{lem_6_1}, 
\ref{lem_6_2}, and \ref{lem_6_3} 
are proved in almost the same way as 
Lemmas 4.1, 4.2, and 4.4 in \cite{Su19_1}, respectively, 
because the difference of condition (K1) between this paper and \cite{Su19_1} 
does not affect the arguments of the proofs. 

\begin{lemma} \label{lem_6_1}
We have $\mathsf{F}(\mathcal{V}_0)=\mathcal{K}(\Theta_E)$. 
\end{lemma}
% 
%%%%%%%%%%%%%
\comment{
\begin{proof}
If $f \in \mathcal{V}_0$, $\mathsf{F}f$ and $\mathsf{F}\mathsf{K}f$ belong to the Hardy space $H^2$.  
On the other hand, we have 
$(\mathsf{F}\mathsf{K}f)(z)=\Theta(z)(\mathsf{F}f)(-z)$ 
by Lemma \ref{lem_3_2}.  
This implies $(\mathsf{F}f)(z)=\Theta(z)(\mathsf{F}\mathsf{K}f)(-z)$ by \eqref{105}.  
Therefore, $\mathsf{F}f$ belongs to $\mathcal{K}(\Theta)$ by \eqref{s502}. %$H^2 \cap \Theta\bar{H}^2$. 
Conversely, if $F \in \mathcal{K}(\Theta)$, there exists $f \in L^2(0,\infty)$ and $g \in L^2(-\infty,0)$ such that $F(z)=(\mathsf{F}f)(z) = \Theta(z) (\mathsf{F}g)(z)$. 
We have $(\mathsf{F}g)(-z) = \Theta(z)(\mathsf{F}f)(-z)$ by \eqref{105} again. 
Here $(\mathsf{F}g)(-z)=(\mathsf{F}g^-)(z)$ for $g^-(x) = g(-x) \in L^2(0,\infty)$, 
and $\Theta(z)(\mathsf{F}f)(-z)=(\mathsf{F}\mathsf{K}f)(z)$ as above. 
Hence $\mathsf{K}f$ belongs to $L^2(0,\infty)$, and thus $f \in \mathcal{V}_0$. 
\end{proof}
} 
%%%%%%%%%%%%%
%

%
\begin{lemma} \label{lem_6_2} 
We have $\mathcal{V}_t\not=\{0\}$ for every $t \in \R$. 
\end{lemma}

\begin{lemma} \label{lem_6_3} Let $z \in \C_+$. Then the integral equations 
\begin{equation} \label{190704_1}
a_z^t + \mathsf{K}\mathsf{P}_t a_z^t = e_z +\mathsf{K}e_z, \qquad 
b_z^t - \mathsf{K}\mathsf{P}_t b_z^t = e_z -\mathsf{K}e_z 
\end{equation}
have unique solutions, where $e_z(x)=\exp(izx)$ 
and $\mathsf{P}_t$ is the projection to $L^2(-\infty,t)$. 
Moreover, for $\Im(z)>c$, %$t<\tau$ and 
\begin{equation} \label{190705_4}
a_z^t(t) = 2 \frac{\frak{A}(t,z)}{E(z)}, \qquad b_z^t(t) = -2i \frac{\frak{B}(t,z)}{E(z)}.
\end{equation}
%
%Here (K5) is necessary but $\tau=\infty$ is unnecessary 
%if we assume that $\Theta$ is an inner function in $\C_+$.  
%
\end{lemma}
Then,  (1) and (2) of Theorem \ref{thm_4}
are proved in the same argument as 
Sections 4.2 and 4.3 of \cite{Su19_1}, respectively.  \hfill $\Box$

\section{Model subspaces for general $t$} \label{section_7} 
%
%%%%%%%%%%%%%%%%%%%%%%%%%%%%%%%%%%%%%%%%%%%%%%%%%%%%%%%%%%%%%%%%%%%%%%%%%%%%%%%%%%%%%%%%
%

As in Section \ref{section_6}, 
we suppose that $E$ satisfies (K1)$\sim$(K5) with $\tau=\infty$ throughout the section.

\begin{theorem} 
Let $A(t,z)$ and $B(t,z)$ be as in Theorem \ref{thm_1} and define 
\begin{equation}\label{s339} 
E(t,z) = A(t,z) - iB(t,z), \qquad 
\Theta(t,z) = \frac{\overline{E(t,\bar{z})}}{E(t,z)}
\end{equation}
for $t \in \R$. Then $\Theta(t,z)$ is an inner function in $\C_+$. 
\end{theorem}
\begin{proof}
Note that $\mathcal{Z} \subset \R$, since $\Theta_E$ is an inner function in $\C_+$ 
by the assumption. 
From Theorem \ref{thm_1}\,(1), we have $\overline{E(t,\bar{z})} = A(t,z) + iB(t,z)$. 
Therefore, $|\Theta(t,z)|=1$ for almost all $z \in \R$. 
Suppose that $t \geq 0$ and put $J(t;z,w)=\overline{E(z)}E(w)j(t;z,w)$. 
Then, \eqref{s106} is transformed as 
\begin{equation} 
J(t;z,w) 
 = \frac{\overline{E(t,z)}E(t,w)-E(t,\bar{z})\overline{E(t,\bar{w})}}{2\pi i(\bar{z}-w)} 
\end{equation}
by using \eqref{s339} and $\overline{E(t,\bar{z})} = A(t,z) + iB(t,z)$. 
On the other hand, we have 
\begin{equation} \label{s342}
\scalebox{0.95}{$\displaystyle
J(t;z,w) - J(s;z,w) 
 = \frac{1}{\pi}\int_{t}^{s}\overline{A(t,z)}A(t,w) \, \frac{1}{\gamma(t)} \, dt 
+ \frac{1}{\pi}\int_{t}^{s}\overline{B(t,z)}B(t,w) \, \gamma(t) \, dt 
$}
\end{equation}
for any $t<s< \infty$ and $z,w \in \C_+$ (\cite[Propositon 2.4]{Su19_1}). 
Taking $w=z \in \C_+$ in \eqref{s342} and then tending $s$ to $\infty$, we have 
\begin{equation} \label{s342_b}
J(t;z,z) 
= \frac{1}{\pi}\int_{t}^{\infty}|A(t,z)|^2 \, \frac{1}{\gamma(t)} \, dt 
+ \frac{1}{\pi}\int_{t}^{\infty}|B(t,z)|^2 \, \gamma(t) \, dt 
\end{equation}
by Theorem \ref{thm_4}\,(2), where $\gamma(t)=m(t)^2>0$. 
This shows that $|\Theta(t,z)|<1$ for any $z \in \C_+$, since the left-hand side 
equals to $(|E(t,z)|^2-|\overline{E(t,\bar{z})}|^2)/(2\pi\Im(z))$. 
Hence $\Theta(t,z)$ is an inner function in $\C_+$. 
For $t \leq 0$, we have 
\begin{equation} \label{190711_1}
E(t,z) = E(z)e^{izt}, \qquad \Theta(t,z) = \Theta_E(z) e^{-2izt}
\end{equation}
by \eqref{s331}. Hence $\Theta(t,z)$ is an inner function in $\C_+$, 
since the factors $\Theta(z)$ and  $e^{-2izt}$ are both inner functions in $\C_+$. 
\end{proof}

Equality \eqref{s342_b} also shows the following. 

\begin{corollary} \label{cor_190705_1} 
$\frak{A}(t,z)$ and $\frak{B}(t,z)$ %of \eqref{190709_3} 
are square-integrable at $t=\infty$ for each $z \in \C_+$. 
%If $E(0)=A(0)\not=0$ in addition, $\gamma(t)^{-1}$ is $L^1$ at $t=\infty$. 
\end{corollary}

Recall that $\int_{\alpha}^{\infty} \gamma(t)^{-1} \, dt<\infty$ is a part of 
the sufficient condition in \cite[Theorem 41]{MR0229011} 
for the existence of the solution of the canonical system for $H(t)={\rm diag}(1/\gamma(t),\gamma(t))$ on $[\alpha,\infty)$.  
We find that it is necessary if $E(0)=A(0)\not=0$. In fact, we have 
$J(t;0,0)  = \pi^{-1}A(0)^2 \int_{t}^{\infty} \gamma(t)^{-1} \, dt $
by taking $z=0$ in \eqref{s342_b}. Therefore, $\gamma(t)^{-1}$ is $L^1$ at $t=\infty$. 
\medskip

By comparing the kernels of $\mathsf{F}(\mathcal{V}_t)$ and $\mathcal{K}(\Theta(t,z))$, 
we obtain the following relation. 
\begin{corollary} \label{cor_190706_1} 
For $t \geq 0$, 
\begin{equation} \label{190706_1} 
\mathsf{F}(\mathcal{V}_t) = \frac{E(t,z)}{E(z)} \mathcal{K}(\Theta(t,z)).
\end{equation}
\end{corollary}

If $t<0$, the space $\mathsf{F}(\mathcal{V}_t)$ is no longer a subspace of $H^2=\mathsf{F}(L^2(0,\infty))$, 
since $\mathcal{V}_t$ contains $L^2(t,-t)$ 
as found in the proof of Lemma \ref{lem_6_2}, 
in particular, $\mathsf{F}(\mathcal{V}_t)$ can not be a model subspace, 
%Therefore equality \eqref{190706_1} can not be extended to $t<0$, 
but the right-hand side of \eqref{190706_1} can be extended to negative $t$'s. 
We found above that $\Theta(t,z)$ is an inner function in $\C_+$. 
The kernel of $\mathcal{K}(\Theta(t,z))$ is 
\[
\aligned 
J(t;z,w) 
& = 
\frac{\overline{A(z)}B(w)-A(w)\overline{B(z)}}{\pi(w-\bar{z})} \cos(t(w-\bar{z})) \\
& \quad 
-(\overline{A(z)}A(w)+\overline{B(z)}B(w)) \frac{\sin(t(w-\bar{z}))}{\pi(w-\bar{z})}
\endaligned 
\]
by \eqref{190711_1}. 
If $\theta_1$ and $\theta_2$ are inner functions in $\C_+$, 
$\mathcal{K}(\theta_1\theta_2) = \mathcal{K}(\theta_2) \oplus \theta_2 \mathcal{K}(\theta_1)$
by \cite[Lemma 2.5]{KW05}. Hence, 
\[ 
\frac{E(t,z)}{E(z)}\mathcal{K}(\Theta(t,z)) 
= e^{itz}( \mathcal{K}(\Theta_E) \oplus \Theta {\rm PW}_{-2t}), 
\]
where ${\rm PW}_a=\mathcal{K}(e^{iaz})$ ($a>0$) is the Paley--Wiener space, 
which consists of the entire functions of exponential type at most $a$ 
the restrictions of which to the real line $\R$ are in $L^2(\R)$. 
Therefore, $\mathsf{F}(\mathcal{V}_t)$ is a shift of a model subspace. 

%
%%%%%%%%%%%%%%%%%%%%%%%%%%%%%%%%%%%%%%%%%%%%%%%%%%%%%%%%%%%%%%%%%%%%%%%%%%%%%%%%%%%%%%%%
%
\section{Related differential equations} \label{section_8} 
%
%%%%%%%%%%%%%%%%%%%%%%%%%%%%%%%%%%%%%%%%%%%%%%%%%%%%%%%%%%%%%%%%%%%%%%%%%%%%%%%%%%%%%%%%
%

From Theorem \ref{thm_5} and the second equation of \eqref{formula_02} and \eqref{formula_02_b}, 
we find that $\Phi(t,x)$ and $\Psi(t,x)$ are characterized as the unique solution 
of the Cauchy problem:
\begin{equation} \label{s904}
\left\{
\aligned
\,& \Phi_t(t,x) + \gamma(t) \Psi_x(t,x)=0, 
\quad \Psi_t(t,x) + \gamma(t)^{-1} \Phi_x(t,x)=0, \\[4pt]
\,& \gamma(t)= \Psi(t,t)/\Phi(t,t) \,( = \Phi(t,t)^{-2} = \Psi(t,t)^2 ), \\
\,& \Phi(0,x) = 1 - \int_{0}^{x} K(y) \, dy, \quad  \Psi(0,x) = 1 + \int_{0}^{x} K(y) \, dy
\endaligned
\right. 
\quad 
\end{equation}
for $(t,x) \in [0,\tau) \times \R$. 
In this formulation, (K5) is understood as a statement about the existence of a global solution.
We should remark that $\gamma(t)$ is not a given function in \eqref{s904} 
different from usual Cauchy problem 
for hyperbolic first-order systems. 
It would be interesting to study the inverse problem for Hamiltonian systems 
in terms of this (unusual) Cauchy problem. 

On the other hand, if we note that $\Phi(t,x)$ and $\Psi(t,x)$ have the second derivative for both variables 
by Propositions \ref{prop_190620_2} and \ref{prop_190706_1}, 
we find that they satisfy the following 
damped wave equations (or wave equations with time-dependent dissipation)  
\begin{equation*}
\left\{
\aligned
\,& \Phi_{tt}(t,x) - \Phi_{xx}(t,x) - 2 \mu(t) \Phi_t(t,x)=0, \\
\,& \Psi_{tt}(t,x) - \Psi_{xx}(t,x) + 2 \mu(t) \Psi_t(t,x)=0 
\endaligned
\right.
\end{equation*}
by the first line of \eqref{s904}, where $2\mu(t)=\gamma(t)'/\gamma(t)$. 
Moreover, definition \eqref{eq_190625_1} derives the Schr{\"o}dinger equations   
\begin{equation*}
\left\{
\aligned
\,& A_{tt}(t,z) + z^2A(t,z) - 2 \mu(t) A_t(t,z)=0, \\
\,& B_{tt}(t,z) + z^2B(t,z) + 2 \mu(t) B_t(t,z)=0.  
\endaligned
\right.
\end{equation*}
These are of course directly proved by Theorem \ref{thm_1}\,(3). 
Taking $z=0$ in Theorem \ref{thm_1}\,(3), we have 
$A(t,0)=A(0)=E(0)$ and $B(t,0)=0$ for each $t<\tau$. 
Therefore, if $E(0)=A(0)\not=0$ and (K5) holds for $\tau=\infty$, 
$j(t;0,z) = (\pi z E(z))^{-1}B(t,z) \to 0$ as $t \to \infty$ 
by Theorem \ref{thm_4}\,(2). 
Hence $B(t,z) \to 0$ as $t \to \infty$, 
and $\Phi(t,x) \to 0$ as $t \to \infty$. 
This fact would be interesting if we recall the following J. Wirth's results. 
He studied the Cauchy
problem for a damped wave equation 
$u_{tt} - u_{xx} + bu_t=0$, 
$u(0,\cdot)=u_1$, $u_t(0,\cdot)=u_2$ 
with positive time-depending dissipation $b=b(t)$. 
He showed that if $tb(t) \to \infty$ as $t \to \infty$, 
$1/b \in L^1(0,\infty)$, 
$u_1 \in W^{s,2}$, and $u_2 \in W^{s-1,2}$,  
the solution $u(t,x)$ tends in $W^{s,2}$ 
to a real analytic function 
$u(\infty,x)=\lim_{t \to \infty}u(t,x)$, 
which is nonzero except at most one $u_2$ for each $u_1$, 
where $W^{s,2}$ is the Sobolev space on $(0,\infty)$ (\cite[p. 76, Result 3]{Wir07}).  
From the result of Wirth, it is naturally asked 
whether $A(t,z)$ or $\Psi(t,z)$ tends to functions as $t \to \infty$. 
This problem is also interesting from the view point of the so-called connection formula 
for solutions of canonical systems: ${}^{\rm t}[A(t,z), B(t,z)] =M(t,s;z){}^{\rm t}[A(s,z), B(s,z)]$, 
but nothing is known in general about the behavior of $A(t,z)$ or $\Phi(t,z)$ when $t \to \infty$. 
If $E(z)$ satisfies $(K1)\sim (K5)$ with $\tau=\infty$, 
belongs  to the P{\'o}lya class, and $E(0)\not=0$, 
it is shown that $\lim_{t \to \infty}A(t,z)=E(0)$ 
by \cite[Theorem 41]{MR0229011}.

%
%---------------------------------------

%---------------------------------------
%

\bigskip \noindent
%Masatoshi Suzuki,\\[5pt]
\\
Department of Mathematics, 
School of Science, \\
Tokyo Institute of Technology \\
2-12-1 Ookayama, Meguro-ku, 
Tokyo 152-8551, JAPAN  \\
Email: {\tt msuzuki@math.titech.ac.jp}

\end{document}